\documentclass[12pt, reqno]{amsart}
\usepackage{amsmath, amsthm, amscd, amsfonts, amssymb, graphicx, color}
\usepackage[bookmarksnumbered, colorlinks, plainpages]{hyperref}

\makeatletter
\@namedef{subjclassname@2010}{%
  \textup{2010} Mathematics Subject Classification}
\makeatother
\newtheorem{theorem}{Theorem}[section]
\newtheorem{lemma}[theorem]{Lemma}

\newtheorem{corollary}[theorem]{Corollary}
\theoremstyle{definition}

\newtheorem{example}[theorem]{Example}

\theoremstyle{remark}
\newtheorem{remark}[theorem]{Remark}
\numberwithin{equation}{section}

\begin{document}
\setcounter{page}{1}
\title[Borg's Theorem]{A Note on Discrete Borg-type Theorems}

\author[V. B. Kiran Kumar, G. Krishna Kumar]{V. B. Kiran Kumar$^{1}$, G. Krishna Kumar$^{2}$}

\address{$^{1}$ Cochin University of Science And Technology, Kerala, India.}
\email{kiranbalu36@gmail.com}

\address{$^{2}$ Department of Mathematics,
BJM Govt. College, Chavara, Kerala, India.}
\email{krishna.math@gmail.com}

%\dedicatory{This paper is dedicated to Professor ABCD}

\subjclass[2010]{Primary 47B35; Secondary 47B36.}

\keywords{Spectrum; Schr$\ddot{\mbox{o}}$dinger operator ; Laurent-Toeplitz operator; Borg-type Theorems}

\date{Received: xxxxxx; Revised: yyyyyy; Accepted: zzzzzz.
\newline \indent $^{1}$ Corresponding author}

\baselineskip=21pt
\vspace*{-1 cm}

\begin{abstract}
We consider the discrete versions of the well-known Borg's Theorem and use simple linear algebraic techniques
to obtain new versions of the discrete Borg-type theorems. To be precise, we prove that the periodic potential
of a discrete Schr$\ddot{\mbox{o}}$dinger operator is almost a constant if and only if the possible spectral gaps
of the operator are of small width. This result is further extended to more general settings and the connection 
to the well-known Ten Martini problem is also discussed. 
\end{abstract}
 \maketitle
\section{Introduction}\label{intro}
Borg's theorem is a classical result in inverse spectral theory which has connections with several problems in quantum mechanics.
In $1946$, G. Borg proved that the periodic potential of a Schr$\ddot{\mbox{o}}$dinger operator is constant almost everywhere
if and only if the essential spectrum of the operator is connected (see \cite{borg}). For mathematicians, this is a very 
interesting result as it determines the nature of the potential from the given spectral data of the operator. 
This result is used as an important tool for various problems in several other branches of science, for example, computing 
the density of a guitar string from its frequency data, to determine the shape of a drum etc. (see \cite{shape}). 
The discrete versions of Borg's theorem were also considered by several mathematicians. In $1975$, H. Flaschka discovered 
the discrete Borg's theorem (see \cite{Fla}). There were several modifications and generalizations of this result, which we refer
as Borg-type theorems. In \cite{kiran}, the authors gave a pure linear algebraic approach to prove the discrete Borg's theorem.
The techniques developed in \cite{kiran} were used to obtain discrete Borg-type theorems in \cite{KSN1}. 
The techniques originally trace back to the classical Bloch/Floquet theory (see \cite{Embree} for an elegant numerical 
illustration of such techniques).

In this article, we prove that {\it the possible spectral gaps of the discrete Schr$\ddot{\mbox{o}}$dinger operator
are of small width if and only if the periodic potential is almost a constant}. This result is further extended to
general Jacobi operators and block Laurent operators. Finally, we connect the results developed with the well-known
Ten Martini problem (see \cite{Avila}). 

Consider the Schr$\ddot{\mbox{o}}$dinger operator $\tilde{A}$ on a suitable subspace of $L^2(\mathbb{R})$ defined by 
\begin{equation}\label{Schrod}
  \tilde{A} u = -\ddot{u}+ v\cdot u,
\end{equation}
where $v$ is an essentially bounded periodic potential. This is an important unbounded operator that arises widely in quantum mechanics. The knowledge of spectrum of this operator is valuable for understanding many physical phenomena in quantum mechanics. There are several classical results which reveals the connection between the nature of spectrum of $\tilde{A}$ and the potential function $v.$ The celebrated Borg's theorem states that there are no gaps in the essential spectrum of $\tilde{A}$ if and only if the periodic potential function $v$ reduces to a constant almost everywhere \cite{borg}.

\subsection{Discrete Schr$\ddot{\mbox{o}}$dinger operator} The simple finite difference approximation will reduce the  Schr$\ddot{\mbox{o}}$dinger operator defined in \eqref{Schrod}, to a bounded self-adjoint operator on  $\ell^2(\mathbb{Z})$, which we call as the discrete Schr$\ddot{\mbox{o}}$dinger operator. The discrete version of Borg's theorem is known due to Flaschka \cite{Fla} and a much simpler proof using basic linear algebraic techniques is done in \cite{kiran}. The key idea used in \cite{kiran} is the identification of the discrete Schr$\ddot{\mbox{o}}$dinger operator as a block Laurent operator (after some scaling and translation). This identification helps us to use the rich theory of spectral analysis of block Laurent operators to the problem. Here we explain this identification briefly.

Consider the Schr$\ddot{\mbox{o}}$dinger operator with periodic potential $v$ and without loss of generality assume that the periodicity of $v$ is $1$. Now approximate the equation $(\ref{Schrod})$ in the interval $[-n,n]$ with $n\in \mathbb{N}\cup \{\infty\}$ by using $p$ equispaced points in each interval $[j,j+1]\subset [-n,n]$. By using the standard difference we have
\[
-\ddot{u}(x_{i, (j)})= \frac{-u(x_{i+1,(j)})+2 u(x_{i,(j)}) - u(x_{i-1,(j)})}{h^2}
\]
with $h= 1/p$. For $j=-n,\cdots,n-1$ we have $\displaystyle{x_{s,(j)}=j+sh}$ where $s=0, \cdots, p-1$. Letting $n=\infty$, this can be treated as an operator denoted by $\tilde{A}$ acting on the sequence space $\ell^2(\mathbb{Z})$ and is defined by
\begin{equation}\nonumber
 \tilde{A}(\{u_n\}_{n\in \mathbb{Z}})=
 \frac{-(u_{n-1}+u_{n+1})+2 u_n}{h^2}+v_nu_n; \{u_n\}_{n\in \mathbb{Z}}\in \ell^2(\mathbb{Z}),
\end{equation}
where the sequence $\{v_n\}_{n\in \mathbb{Z}}$ is obtained as the values of the periodic function $v$ at $p$ equispaced points in an interval of length $1$. The periodicity of $v$ will imply that the sequence $\{v_n\}_{n\in \mathbb{Z}}$ is also periodic with period $p$. The matrix representation of $\tilde{A}$ with respect to the standard basis of $\ell^2(\mathbb{Z})$ is obtained as a tridiagonal matrix (up to the scaling factor $h^2$) as 
\begin{equation}\nonumber
\tilde{A}= \left[ \begin{array}{cccccccccc}
 \ddots & \ddots  \\
  \ddots & \ddots & \ddots & \\
    & -1 & 2+h^2v(x_{0, (1)}) & -1 & \\
     &   & \ddots &\ddots& \ddots \\
     &   &    &-1 & 2+h^2v(x_{p-1, (1)})  & -1\\
      &  &     &      & -1 & 2+h^2v(x_{0, (2)}) & -1 \\
      &  &     &      &    & \ddots & \ddots & \ddots \\
     &   &    &     &    &  & \ddots 	& \ddots
     \end{array} \right].
\end{equation}
Since $v$ is periodic, we have $v(x_{s,(j)})= v(x_{s,(j+1)})$ for all $s= 0, \cdots, p-1$. Thus the sequence appearing in the main diagonal becomes periodic with period $p$. When $n$ is finite, the resulting  matrix of size $np$ is the truncation of the bi-infinite matrix reported above. When $n=\infty,$ up to some scaling and translation by a scalar multiple of the identity operator, this operator can be identified as the block Laurent operator defined by 
\begin{equation}\label{anm2}
 A:= \begin{bmatrix}\cdot& \cdot& \cdot& \cdot& \cdots& \cdot&\cdot\\
 \cdot& A_{1}&A_0& A_{-1}&\cdots& \cdot&\cdot\\
 \cdot&\cdot&A_{1}& A_{0}& A_{-1}&\cdot&\cdot\\
 \cdot&\cdot&\cdot& A_{1}& A_{0}& A_{-1}& \cdot\\
 \cdot&\cdot&\cdots&\cdots& \cdot&\cdot&\cdot \end{bmatrix}
\end{equation}
where 
\[
A_1= \begin{bmatrix}0&0&\cdots&1\\0&0&\cdots&0\\\cdot&\cdot&\cdots& \cdot\\
       0&0&\cdots&0
      \end{bmatrix}\;\;\;,\;\;\;
A_{-1}= \begin{bmatrix}0&0&\cdots&0\\0&0&\cdots&0\\\cdot&\cdot&\cdots& \cdot\\
       1&0&\cdots&0
      \end{bmatrix} \;\;\;\mbox{and}\;\;\;
A_0= \begin{bmatrix}v_1&1&\cdots&0\\1&v_2&\cdots&0\\\cdot&\cdot&\cdots& \cdot\\
       0&0&\cdots&v_p
      \end{bmatrix}. 
\]    
Thus after some suitable translation and scaling the discrete Schr$\ddot{\mbox{o}}$dinger operator $\tilde{A}$ becomes the operator $A$ defined in (\ref{anm2}). It is worthwhile to notice that \eqref{anm2} gives the matrix representation of the bounded self-adjoint operator $A:\ell^2(\mathbb{Z})\rightarrow \ell^2(\mathbb{Z})$ defined by 
\begin{equation}\label{discsr}
A(x_n)=(x_{n-1}+x_{n+1}+v_nx_n)
\end{equation}
with respect to the standard orthonormal basis. We recall the notion of Toeplitz operators below.

\subsection{Toeplitz operators and sequences}\label{ssec:toeplitz}
Given a $p\times p$ matrix-valued integrable function $f$ defined on 
$(-\pi , \pi)$, the $p\times p$ matrices $f_j$, $j\in \mathbb Z$, represent the Fourier coefficients of $f$ defined as 
\begin{equation*}
f_j(\theta) = \frac{1}{2\pi}\int_{-\pi}^{\pi}f(x)e^{- \hat{i}j\, \theta }dx, \ \ \ j = 0,\pm 1, \pm 2, \ldots.
\end{equation*}
Then for $n$ being a nonnegative integer number or $\infty$ we define $T_n(f) $ the Toeplitz matrix or operator of size  $n$
generated by $f$ via the relations 
\[
(T_n(f))_{i,j}= {f_{i-j}}, \ \ \ \ i,j=1,\ldots,n.
\]
Here the integration of the matrix valued function turns out to be the entry wise integration.  
 When $n=\infty$ the Toeplitz operator $T_n(f)$ is simply written as $T(f)$ while the symbol $L(f)$  denotes the doubly infinite
Toeplitz matrix with $(L(f))_{i,j}= {f_{i-j}}$, $i,j\in \mathbb Z$. Furthermore by $\{ T_n(f) \}$ we indicate the Toeplitz matrix-sequence generated by $f$, with $T_n(f)$ of finite order.

Let $f$ be a  continuous and Hermitian $p\times p$ matrix-valued function on the unit circle, and let $\lambda_1(f(.))\ge \cdots \ge \lambda_p(f(.))$ denote its eigenvalues. Then
it is well known that the essential spectrum of $L(f)$ and the essential spectrum of $T(f)$ coincide with the
union of the ranges of the eigenvalues $\lambda_1(f(\cdot))\ge \cdots \lambda_p(f(\cdot))$, that is
\begin{equation}\label{bottcher-id}
\sigma_{\hbox{ess}}(L(f))= \sigma_{\hbox{ess}}(T(f)) = \bigcup_{j=1}^p \left[\inf_\theta(\lambda_j(f(\theta))),\sup_\theta(\lambda_j(f(\theta)))\right].
\end{equation}
For the latter result, see proposition $2.29(a)$ of the book \cite{bottcher}.

Now we consider the matrix valued symbol $f$ of the Laurent operator $A$, obtained from the matrix representation \eqref{anm2}. 
Then $f$ is given by 
$f(\theta)=A_0+A_{-1}e^{-i\theta}+A_{1}e^{i\theta}, \,\,-\pi<\theta\leq \pi.$
 Therefore we have for $-\pi<\theta\leq \pi,$
\[
 f(\theta)= \begin{bmatrix}v_1&1&0&\cdots&e^{i\theta}\\
 1&v_2&1&\cdots&0\\\cdot&\cdot&\cdot&\cdots&\cdot&\\
 e^{-i\theta}&\cdot&\cdot&\cdots&v_{p}\end{bmatrix}.
\]
From \eqref{bottcher-id}, we have
\[
\sigma(A)=\sigma(f)= \bigcup_{-\pi<\theta\leq \pi}\sigma(f(\theta)).
\]
 This identity will play an important role in the proof of our main results. 
 It is worthwhile to notice that the symbol $f$ can be chosen to be any member of the set $\{f_k;k=0,1,2\ldots p-1\}$, 
 where $f_k$ is defined by
\[
  f_k(\theta)= \begin{bmatrix}v_{k+1}&1&0&\cdots&e^{i\theta}\\ 1&v_{k+2}&1&\cdots&0\\\cdot&\cdot&\cdot&\cdots&\cdot&\\
  e^{-i\theta}&\cdot&\cdot&\cdots&v_{k+p}\end{bmatrix},\; k= 0, 1, \cdots, p-1.
\]
Because of the periodicity of the sequence $\{v_n\}_{n\in \mathbb{Z}}$, we have the freedom to choose any one of the $f_k^{'}s$.
\section{Discrete Borg-type theorems}\label{mainres}
Recall that the discrete Borg's theorem states that the spectrum of the discrete Schr$\ddot{\mbox{o}}$dinger operator 
with real-valued periodic potential is connected if and only if the periodic potential reduces to a constant. Here we consider the 
case when the spectrum need not be connected. That is the spectrum is the union of intervals with some possible gaps (spectral gaps).
In this section we prove the new versions of discrete Borg's theorem. We prove that
the possible spectral gaps of the discrete Schr$\ddot{\mbox{o}}$dinger operator with real-valued 
periodic potential are of small width whenever the potential is almost a constant.
\begin{theorem}\label{pseudo borg1}
Let $A$ be the discrete Schr$\ddot{\mbox{o}}$dinger operator with real periodic potential sequence $\{v_n\}_{n\in \mathbb{Z}}$ 
defined in  \text{\textnormal{(\ref{discsr})}} and assume that the maximum width of the spectral gap of $A$ 
is less than $\epsilon$ for some $\epsilon>0$. Then there exists a constant $c$ such that 
$\displaystyle{\sup_{n\in \mathbb{Z}}|v_n-c|\leq \epsilon(p-1)}$, where $p$ is the period of the potential sequence $v_n$.
\end{theorem}
\begin{proof}
Consider the symbols $f_k$ of $A$ defined by 
\[
  f_k(\theta)= \begin{bmatrix}v_{k+1}&1&0&\cdots&e^{i\theta}\\ 1&v_{k+2}&1&\cdots&0\\\cdot&\cdot&\cdot&\cdots&\cdot&\\
  e^{-i\theta}&\cdot&\cdot&\cdots&v_{k+p}\end{bmatrix},\; k= 0, 1, \cdots, p-1.
\]
Let $\sigma(f_k(\theta))= \{\lambda_1(f_k(\theta))\geq \lambda_2(f_k(\theta))\geq \cdots \geq \lambda_p(f_k(\theta))\}$ for each $-\pi<\theta\leq \pi$.
Now we consider the sub matrices $J_k$ and $J_{k+1}$ of $f_k(\theta)$ defined by
 \[
  J_k:= \begin{bmatrix}v_{k+1}&1&0&\cdots&0\\1&v_{k+2}&1&\cdots&0\\
         \cdot&\cdot&\cdot&\cdots&\cdot\\0&0&0&\cdots&v_{k+p-1}
        \end{bmatrix},\,\,   J_{k+1}:= \begin{bmatrix}v_{k+2}&1&0&\cdots&0\\1&v_{k+3}&1&\cdots&0\\
         \cdot&\cdot&\cdot&\cdots&\cdot\\0&0&0&\cdots&v_{k+p}
        \end{bmatrix}.
 \]
Notice that $J_k$ and $J_{k+1}$ are Jacobi matrices. Let 
\[
 \sigma(J_k)= \{\mu_{1,k}\geq \mu_{2,k}\geq\cdots\geq \mu_{p-1, k}\},\;\;\; \sigma(J_{k+1})= \{\mu_{1,k+1}\geq \mu_{2,k+1}\geq\cdots\geq \mu_{p-1, k+1}\}. 
\]
The Cauchy interlacing properties for eigenvalues of Hermitian matrices gives
\[
 \lambda_1(f_k(\theta))\geq \mu_{1, k}\geq\lambda_2(f_k(\theta))\geq \cdots\geq \lambda_{p-1}(f_k(\theta))\geq \mu_{p-1, k}\geq \lambda_p(f_k(\theta)),
\]
\[
 \lambda_1(f_k(\theta))\geq \mu_{1, k+1}\geq\lambda_2(f_k(\theta))\geq \cdots\geq \lambda_{p-1}(f_k(\theta))\geq \mu_{p-1, k+1}\geq \lambda_p(f_k(\theta)).
\]
The above inequalities are true for all $-\pi<\theta\leq \pi$. Since the spectral gaps are of width less than $\epsilon,$ 
$ \bigcup_{-\pi<\theta\leq\pi} \sigma(f_k(\theta))+(-\frac{\epsilon}{2},\frac{\epsilon}{2}) $ is connected.
Therefore we have 
\begin{equation}\label{ineq1}
 |\lambda_{j,k}^--\lambda_{j+1, k}^+|\leq \epsilon \text{ for each }k= 0, 1, \cdots, p-1\; \text{ and }\; j= 1, 2,\cdots p,
\end{equation}

 where $ \lambda_{j,k}^+:= \max_{\theta}\lambda_j(f_k(\theta))\;\;\; \mbox{and}\;\;\;
 \lambda_{j,k}^-:= \min_{\theta}\lambda_j(f_k(\theta)).$
We claim that $|\mu_{j, k'}- \mu_{j, k''}|\leq \epsilon$ for all $j= 1, 2, \cdots, p$ and for any two different $k', k''$. 
If this is not the case, then we have some $k', k''$ and $j$ such that $|\mu_{j, k'}- \mu_{j, k''}|> \epsilon.$
From the interlacing property mentioned above and from the fact that we can choose any of the symbol $f_k,$ we get
the contradiction to \eqref{ineq1}. Thus the $j^{th}$ eigenvalue of $J_{k'}$ and $J_{k''}$ are atmost separated by $\epsilon$ distance.
Consider $k'= 0$ and $k''= 1$, we have
\[
 |\mbox{Trace}(J_0)- \mbox{Trace}(J_1)|\leq \epsilon(p-1).
\]
Therefore we get
\[
 \left|\sum_{j= 1}^{p-1}v_j- \sum_{j=2}^pv_j\right|\leq \epsilon(p-1).
\]
Thus $ |v_1- v_p|\leq \epsilon(p-1).$
In general by considering $J_0$ and $J_{i-1}$ we get
\[
 \left|\sum_{j= 1}^{p-1}v_j- \sum_{j=i}^{p-2+i}v_j\right|= |\mbox{Trace}(J_0)- \mbox{Trace}(J_{i-1})|\leq \epsilon(p-1),
\]
and therefore $ |v_{i-1}- v_p|\leq \epsilon(p-1)$ for all $i= 2, 3, \cdots, p$. Hence by choosing $c= v_p$ we get the desired result.
\end{proof}
The following theorem gives the converse of the above result.
\begin{theorem}\label{pseudo borg2}
Let $A$ be the discrete Schr$\ddot{\mbox{o}}$dinger operator with real periodic potential sequence $\{v_n\}_{n\in \mathbb{Z}}$
defined by  \text{\textnormal{(\ref{discsr})}}. 
If there exists a constant $c$ such that $\displaystyle{\sup_{n\in \mathbb{Z}}|v_n- c|\leq \epsilon }$ for some $\epsilon> 0$, 
then  the possible spectral gaps of $A$ are of maximum size less than or equal to $2\epsilon.$
\end{theorem}
\begin{proof}
Since there exists a constant $c$ such that $\displaystyle{\sup_{n\in \mathbb{Z}}|v_n-c|\leq \epsilon }$, we have
$ |v_{i}- v_{j}|\leq 2\epsilon,$ for every $i,j \in  \mathbb{Z}.$ Thus the matrix valued symbol $f(\theta)$ becomes
\[
 f(\theta)= \begin{bmatrix}v_1&1&0&\cdots&e^{i\theta}\\
 1&v_2&1&\cdots&0\\\cdot&\cdot&\cdot&\cdots&\cdot&\\
 e^{-i\theta}&\cdot&\cdot&\cdots&v_{p}\end{bmatrix}, 
\]
with $|v_1-v_j|\leq 2\epsilon$ for all $j= 2, 3, \cdots, p$. Thus $v_1- v_j= r_j$ with $\lvert r_j\rvert\leq 2\epsilon$.
Then $f(\theta)$ can be written as
\[
 f(\theta)= \begin{bmatrix}v_1&1&0&\cdots&e^{i\theta}\\
 1&v_1- r_2&1&\cdots&0\\\cdot&\cdot&\cdot&\cdots&\cdot&\\
 e^{-i\theta}&\cdot&\cdot&\cdots&v_{1}-r_p\end{bmatrix}. 
\]
If we choose 
\[
 f_1(\theta)= \begin{bmatrix}v_1&1&0&\cdots&e^{i\theta}\\
 1&v_1&1&\cdots&0\\\cdot&\cdot&\cdot&\cdots&\cdot&\\
 e^{-i\theta}&\cdot&\cdot&\cdots&v_{1}\end{bmatrix},
\]
then 
\[
 f_1(\theta)- f(\theta)= \begin{bmatrix}0&0&0&\cdots&\\
 0&r_2&0&\cdots&0\\\cdot&\cdot&\cdot&\cdots&\cdot&\\
 0&\cdot&\cdot&\cdots&r_p\end{bmatrix}
\]
Thus we have $\|f_1(\theta)- f(\theta)\|\leq 2\epsilon$. By the discrete version of Borg's theorem proved in \cite{kiran}, 
the spectrum of the block Laurent operator $A_1$ corresponding to the matrix valued symbol $f_1(\cdot)$ is connected.

We will show that $\sigma(A) +(-\epsilon,\epsilon) =\bigcup_{-\pi<\theta\leq \pi}\sigma(f(\theta)) +(-\epsilon,\epsilon) $ is connected.

Let $A, A+E \in \mathbb{C}^{N\times N}$ be Hermitian matrices then $|\lambda_j(A)- \lambda_j(A+E)|\leq \|E\|$ where
$\lambda_j(A)$ denotes the $j^{th}$ eigenvalue of $A$ (see \cite{golub} for eg.).
 Therefore $|\lambda_j(f(\theta))- \lambda_j(f_1(\theta))|\leq 2\epsilon$ for all $j= 1, 2, \cdots, p$. 
 This together with the connectedness of  $\sigma(A_1)=\displaystyle{\bigcup_{-\pi<\theta\leq \pi}\sigma(f_1(\theta))}$ 
 implies that $\sigma(A) +(-\epsilon,\epsilon)= \bigcup_{-\pi<\theta\leq \pi} \sigma(f(\theta))+(-\epsilon,\epsilon)$ is also connected. 
 Hence the possible spectral gaps can have size at most $2\epsilon.$
\end{proof}
\begin{remark}
 It is easy to see that the discrete Borg's theorem follows from the above theorems. 
\end{remark}
\begin{remark}
The bound obtained in Theorem \ref{pseudo borg1} involves the parameter $p.$ This will create some computational obstacles when we deal with sequence of operators.
\end{remark}
\begin{example}
 Consider the $5\times 5$ matrix $f(\theta)$ defined by
 \[
  f(\theta)= \begin{bmatrix}1&1&0&0&e^{i\theta}\\1&1.1&1&0&0\\0&1&1.2&1&0\\0&0&1&1.3&1\\
  e^{-i\theta}
&0&0&1&1.4\end{bmatrix}.
 \]
Here $f(\theta)$ will be the symbol matrix corresponding to some discrete Schr$\ddot{\text{\textnormal{o}}}$dinger operator
$A$ (after some suitable scaling and translation). We have $v_1= 1, v_2= 1.1, v_3= 1.2, v_4= 1.3, v_5= 1.4$ 
and $|v_i- 1.2|\leq 0.2$ for all $i= 1,\cdots, 5$. Figure $1$ is $\sigma(A)$ which is disconnected. However, Figure $2$
shows that the set $\Lambda_{0.2}(A)=\sigma(A)+B(0,0.2)$ is connected, where $B(0,0.2)$ is the disc of radius $0.2$ about the origin.
Hence the spectral gap size is less than $0.4.$
\begin{figure}[ht]
\begin{center}
\includegraphics[width=15cm]{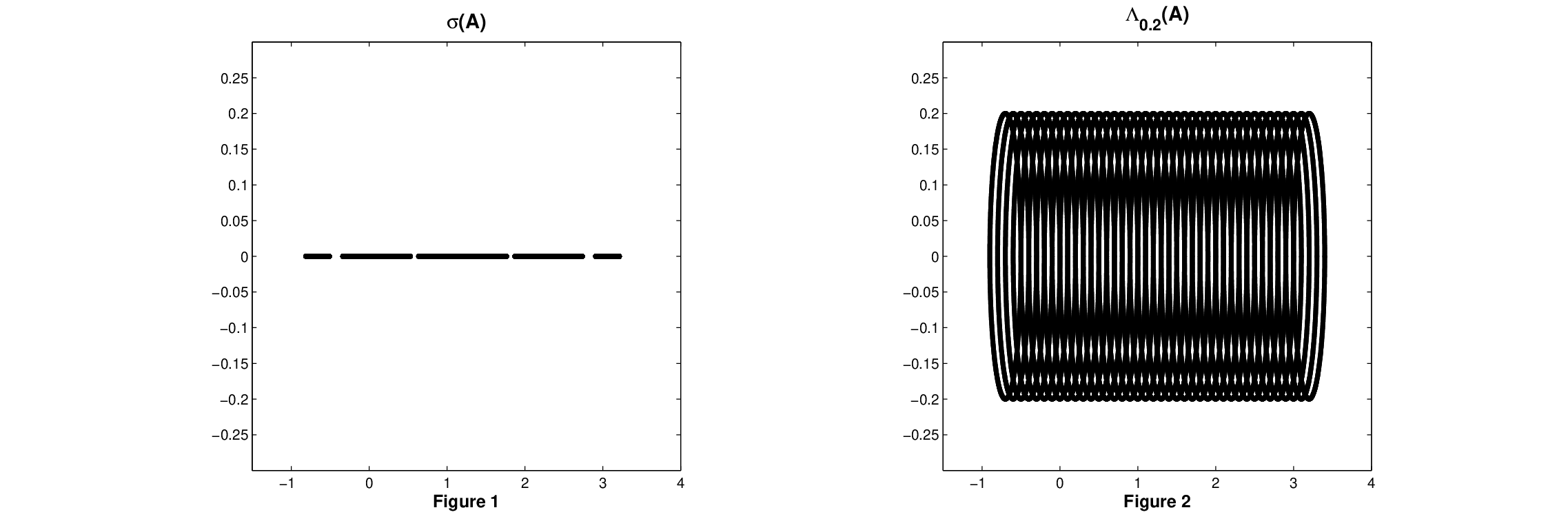}
\end{center}
\end{figure}
\end{example}
\begin{example}
 Consider the $10\times 10$ matrix $f(\theta)$ defined by
 \[
  f(\theta)= \begin{bmatrix}0&1&0&\cdots&\cdots&e^{i\theta}\\1&0.05&1&0&\cdots&0\\0&1&0.1&1&\cdots&0\\ \cdots&\cdots&\cdots&\cdots&\cdots&\cdots\\
e^{-i\theta}&0&0&0&\cdots&0.45\end{bmatrix}_{10\times 10}.
 \]
Here we have $v_1= 0, v_2= 0.05, \cdots, v_{10}= 0.45$ and $|v_i- 0.225|\leq 0.225$ for all $i= 1,\cdots, 10$.
Figure $3$ is $\sigma(A)$  and Figure $4$ is $\Lambda_{0.225}(A)=\sigma(A)+B(0,0.225)$. 
Similar observations can be made here.
\begin{figure}[ht]
\begin{center}
\includegraphics[width=15cm]{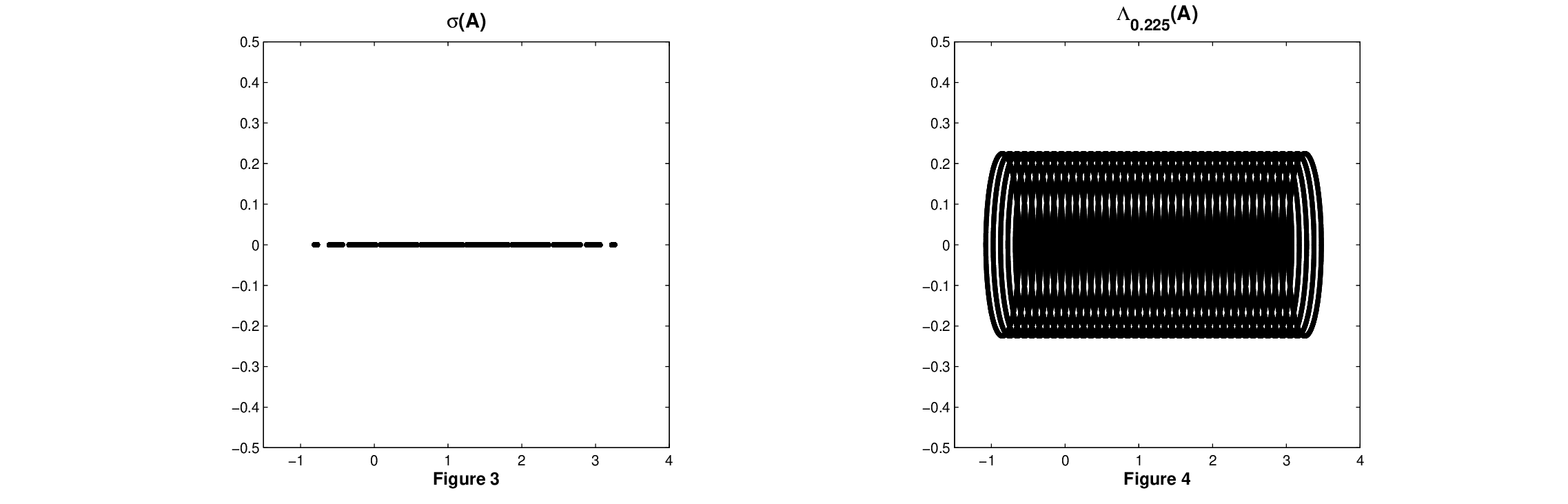}
\end{center}
\end{figure}
\end{example}
\section{Further Generalizations}\label{furthergen}
In this section we consider the one dimensional discrete Schr$\ddot{\mbox{o}}$dinger operator with periodic potential and 
variable coefficients. We generalize the results obtained in the last section to such operators and to 
a much general class of block Laurent operators.
\subsection{Schr$\ddot{\mbox{o}}$dinger operator with variable coefficient}
Consider the one dimensional Schr$\ddot{\mbox{o}}$dinger operator with real-valued periodic potential and variable coefficient 
defined by \[\tilde{A}(u)= -{\frac{d}{dx}}\left(a\cdot{\frac{d}{dx}}{u}\right)+ v\cdot u,\] with $a$ being positive and periodic 
with the same period as $v$. As similar to one dimensional Schr$\ddot{\mbox{o}}$dinger operator with constant coefficient,
the finite difference approximation will lead $\tilde{A}$ to a general Jacobi operator (after some suitable scaling and translation). This can be represented as a bi-infinite symmetric matrix in the following form.
\begin{equation}\label{anm3}
A_{n,\alpha}(p)= \left[ \begin{array}{cccccccccc}
 \ddots & \ddots  \\
  \ddots & \ddots & \ddots & \\
    & -\alpha_{p-1} & \gamma_0 & -\alpha_{0}& \\
     &   & \ddots &\ddots& \ddots \\
     &   &    &-\alpha_{p-2} & \gamma_{p-1}  & -\alpha_{p-1}\\
      &  &     &      & -\alpha_{p-1}& \gamma_0 & -\alpha_{0} \\
      &  &     &      &    & \ddots & \ddots & \ddots \\
     &   &    &     &    &  & \ddots 	& \ddots
     \end{array} \right],
\end{equation}
with $\gamma_s=\alpha_s+\alpha_{(s+1)\ \hbox{mod}\ p}+ h^2v(x_{s;(j)})$ and $\alpha_s=a(x_{s+{1/2}},j),x_{s+{1/2},j}=j+h(p)(s+{1/2})$.
Observe that the resulting structure, up to the sign, represents the case of general $p-$periodic Jacobi matrices.
As we did earlier, this can be identified as the block Laurent operator defined by
\[
 A:= \begin{bmatrix}\cdot& \cdot& \cdot& \cdot& \cdots& \cdot&\cdot\\
 \cdot& A_{1}&A_0& A_{-1}&\cdots& \cdot&\cdot\\
 \cdot&\cdot&A_{1}& A_{0}& A_{-1}&\cdot&\cdot\\
 \cdot&\cdot&\cdot& A_{1}& A_{0}& A_{-1}& \cdot\\
 \cdot&\cdot&\cdots&\cdots& \cdot&\cdot&\cdot \end{bmatrix}
\]
where each $A_0, A_1, A_{-1}$  are $p\times p$ matrices defined respectively by 
\[
 A_0= \begin{bmatrix} v_1 & a_1&0 &\cdots &0\\ a_1&v_2& a_2&\cdots &0\\\cdot&\cdot&\cdot&\cdots&\cdots\\\cdot&\cdot&\cdot&\cdots&a_{p-1}\\ 0&0&0& a_{p-1}&v_p\end{bmatrix},
 \;\; A_1= \begin{bmatrix}0&0&\cdots&a_p\\0&0&\cdots&0\\\cdot&\cdot&\cdots& \cdot\\
       0&0&\cdots&0
      \end{bmatrix},\;\; 
A_{-1}= \begin{bmatrix}0&0&\cdots&0\\0&0&\cdots&0\\\cdot&\cdot&\cdots& \cdot\\
       a_p&0&\cdots&0
      \end{bmatrix}.
      \]
Notice that this infinite matrix gives the matrix representation of the bounded self-adjoint operator $A:\ell^2(\mathbb{Z})\rightarrow \ell^2(\mathbb{Z})$ defined by 
\begin{equation}\label{jacobl2}
A(x_n)=(a_nx_{n-1}+a_nx_{n+1}+v_nx_n)
\end{equation}
with respect to the standard orthonormal basis.
The matrix valued symbol of $A$ will have the following form for $-\pi<\theta\leq \pi$;
\[
 f(\theta)=A_0+A_1e^{i\theta}+A_{-1}e^{-i\theta} =\begin{bmatrix}v_1&a_1&0&\cdots&a_{p}e^{i\theta}\\
 a_{1}&v_2&a_{2}&\cdots&0\\\cdot&\cdot&\cdot&\cdots&\cdot&\\
 a_{p}e^{-i\theta}&\cdot&\cdot&\cdots&v_{p}\end{bmatrix}.
\]
Assume that $a_{j+p}=a_j>0$ for each $j$. We also have the important identity
\[
\sigma(A)=\bigcup_{-\pi<\theta\leq \pi}\sigma(f(\theta)).
\]
Theorem \ref{pseudo borg1} and \ref{pseudo borg2} will have their natural generalizations to the Jacobi case. The proof techniques are almost the same. However we present it below for the sake of completion .
\begin{theorem}\label{pseudo borgjacobi}
Let $A$ be the discrete Schr$\ddot{\mbox{o}}$dinger operator with variable coefficients defined by \eqref{jacobl2}.
If the possible spectral gaps of $A$ have size less than $\epsilon$ for some $\epsilon > 0,$ then there exists a constant
$c$ such that $\displaystyle{\sup_{k\in \mathbb{N}}|v_k-c|\leq \epsilon(p-1)}$.
\end{theorem}
\begin{proof}
The proof is merely an imitation of the proof of Theorem \ref{pseudo borg1}. Here the matrix-valued symbols are
\[
 f_k(\theta)= \begin{bmatrix}v_{k+1}&a_{k+1}&0&\cdots&a_{k+p}e^{i\theta}\\
 a_{k+1}&v_{k+2}&a_{k+2}&\cdots&0\\\cdot&\cdot&\cdot&\cdots&\cdot&\\
 a_{k+p}e^{-i\theta}&\cdot&\cdot&\cdots&v_{k+p}\end{bmatrix}.
\]
It suffices to consider the sub matrices
 \[
  J_k:= \begin{bmatrix}v_{k+1}&a_{k+1}&0&\cdots&0\\a_{k+1}&v_{k+2}&a_{k+1}&\cdots&0\\
         \cdot&\cdot&\cdot&\cdots&\cdot\\0&0&0&\cdots&v_{k+p-1}
        \end{bmatrix},\; k= 0, 1, \cdots, p-1,
 \]
 and apply Cauchy interlacing theorem for the eigenvalues of $J_k.$ Considering $J_0$ and $J_{i-1}$ for $i=2,3\ldots p,$ we get
 \[
  |v_{i-1}- v_p|= \left|\sum_{j= 1}^{p-1}v_j- \sum_{j=i}^{p-2+i}v_j\right|= |\mbox{Trace}(J_0)- \mbox{Trace}(J_i)|\leq \epsilon(p-1).
 \]
Hence the proof is completed.
\end{proof}

\begin{theorem}\label{pseudo borgjacobi1}
Let $A$ be the discrete Schr$\ddot{\mbox{o}}$dinger operator with variable coefficients defined by \eqref{jacobl2}. 
If there exist $c_1, c_2\in \mathbb{R}$ such that $\displaystyle{\sup_{n\in \mathbb{Z}}|v_n-c_1|\leq \epsilon}$ and 
$\displaystyle{\sup_{n\in \mathbb{Z}}|a_n-c_2|\leq \epsilon}$ for some $\epsilon> 0$. Then the maximum size of the spectral gaps of $A$
is less than or equal to  $3\epsilon.$
\end{theorem}
\begin{proof}
Consider $f_0(\theta)$, the same $p\times p$ self-adjoint matrix defined in Theorem \ref{pseudo borgjacobi}.
We have $\sigma(A)+ (-\frac{3}{2}\epsilon,\frac{3}{2}\epsilon)=\bigcup_{-\pi<\theta\leq \pi} \sigma(f_0(\theta))+ (-\frac{3}{2}\epsilon,\frac{3}{2}\epsilon)$. 
Define
\[
 f(\theta)= \begin{bmatrix}v_{1}&a_{1}&0&\cdots&a_{1}e^{i\theta}\\
 a_{1}&v_{1}&a_{1}&\cdots&0\\\cdot&\cdot&\cdot&\cdots&\cdot&\\
 a_{1}e^{-i\theta}&\cdot&\cdot&\cdots&v_{1}\end{bmatrix}.
\]
Let $A_1$ be the discrete Schr$\ddot{\mbox{o}}$dinger operator corresponding to $f$. We have 
\[
 |\lambda_j(f(\theta))- \lambda_j(f_0(\theta))| \leq \|f(\theta)-f_0(\theta)\|\leq 3\epsilon.
\]
where $\lambda_j(f(\theta))$ denotes the $j^{th}$ eigenvalue of $f(\theta)$.
Since $\displaystyle{\sigma(A_1)= \bigcup_{-\pi<\theta\leq \pi} \sigma(f(\theta))}$ is connected (see \cite{kiran}),
the set $\sigma(A)+ (-\frac{3}{2}\epsilon,\frac{3}{2}\epsilon)$ becomes connected. Therefore we have the desired result. 
\end{proof}
\subsection{General block Laurent operators}
Now we consider a general class of block Laurent operators and prove discrete Borg-type theorem for the such class of operators.
The proof techniques are not much different, but we have to take care of the convergence issues of the matrix entries.
Theorem $4.1$ of \cite{KSN1} can be modified in the following way.
\begin{theorem}\label{pseudo Borggn}
Let $A$ be the bounded operator defined by the block Toeplitz-Laurent matrix
\[
 A = \left[ {\begin{array}{*{20}c}
   {\ddots } & {\ddots } & {} & {} & {} & {} & {} & {} & {} & {}  \\
   {\ddots} & {A_0 } & {A_{ - 1} } & {A_{ - 2}} & {} & {\ldots} & {A_{ - N}} & {\ldots} & {} & {}  \\
   {} & {A_1 } & {A_0 } & {A_{ - 1} } & {A_{ - 2}} & {} & {\ldots} & {A_{ - N}} & {\ldots} & {}  \\
   {} & {A_{ 2}} & A_1 & A_0 &A_{ - 1} & {A_{ - 2}} & {} & {\ldots} & {A_{ - N}} & {\ldots}  \\
   {} & {} & {A_{ 2}} & A_1 & A_0 & A_{ - 1} & {A_{ - 2}} & {} & {\ldots} & {A_{ - N}}  \\
   {\ldots} & {A_{ N}} & {\ldots} & {A_{ 2}} & A_1 & A_0 &A_{ - 1} & {A_{ - 2}} & {}   \\
   {} & {\ldots} & {A_{ N}} & {\ldots} & {A_{ 2}} & A_1 & A_0 & A_{ - 1} & {A_{ - 2}} & {}  \\
   {} & {} & {\ldots} & {A_{ N}} & {\ldots} & {A_{ 2}} & A_1 & A_0 & A_{ - 1} & {A_{ - 2}}  \\
   {} & {} & {} & {\ldots} & {A_{ N}} & {\ldots} & {A_{ 2}} & {A_1 } & {A_0 } & {\ddots }  \\
   {} & {} & {} & {} & {} & {} & {} & {} & {\ddots } & {\ddots}  \\
\end{array}} \right],
\]
where 
\[
 A_0  = \left[ {\begin{array}{*{20}c}
   {v_1 } & 1 & {} & {} & {} & {a_0}  \\
   1 & {v_2 } & 1 & {} & {} & {}  \\
   {} & 1 & . & . & {} & {}  \\
   {} & {} & . & . & . & {}  \\
   {} & {} & {} & . & . & 1  \\
   {a_0} & {} & {} & {} & 1 & {v_p }  \\
\end{array}} \right], \ \ \
 A_{ k}  = \left[ {\begin{array}{*{20}c}
   {} & {} & {} & {} & {} & {a_k}  \\
   {} & {} & {} & {} & {} & {}  \\
   {} & {} & {} & {} & {} & {}  \\
   {} & {} & {} & {} & {} & {}  \\
   {} & {} & {} & {} & {} & {}  \\
   {} & {} & {} & {} & {} & {}  \\
\end{array}} \right]={A_{ -k}}^T,
\]
such that $v_1\leq v_2 \ldots \leq v_p$ and $\sum_k \left|a_k\right|< \infty$.
 If the spectral gaps of $A$ are of size less than $\epsilon>0,$ then there exists a constant $c$ such that 
 $\displaystyle{\sup_{k}|v_k-c|\leq \epsilon(p-1)}$.
\end{theorem}

\begin{proof}
The matrix-valued symbol associated with the block Toeplitz-Laurent operator $A$ is
 \[
 \tilde{f}\left( \theta  \right) = \left[ {\begin{array}{*{20}c}
   {v_1 } & 1 & {} & {} & {} & {f\left(\theta \right) }  \\
   1 & {v_2 } & 1 & {} & {} & {}  \\
   {} & 1 & . & . & {} & {}  \\
   {} & {} & . & . & . & {}  \\
   {} & {} & {} & . & . & 1  \\
   {\bar{f}\left(\theta \right) } & {} & {} & {} & 1 & {v_p }  \\
\end{array}} \right].
\]
where $f\left(\theta \right)= \sum_k a_ke^{ ik\theta }$. 
As in the previous theorem, we consider the sub matrices 
\[
\begin{array}{l}
 P_1  = \left[ {\begin{array}{*{20}c}
   {v_1 } & 1 & {} & {} & {} & {}  \\
   1 & {v_2 } & 1 & {} & {} & {}  \\
   {} & 1 & . & . & {} & {}  \\
   {} & {} & . & . & . & {}  \\
   {} & {} & {} & . & . & 1  \\
   {} & {} & {} & {} & 1 & {v_{p-1} }  \\
\end{array}} \right],  
 \,\,\,\,\,\,\,\,\,\,\, 
 P_2  = \left[ {\begin{array}{*{20}c}
   {v_2 } & 1 & {} & {} & {} & {}  \\
   1 & {v_3 } & 1 & {} & {} & {}  \\
   {} & 1 & . & . & {} & {}  \\
   {} & {} & . & . & . & {}  \\
   {} & {} & {} & . & . & 1  \\
   {} & {} & {} & {} & 1 & {v_p }  \\
\end{array}} \right]. \\ 
 \end{array}
\]
If any of their eigenvalues are different, say $\lambda_j(P_1)<\lambda_j(P_2)$, then by Cauchy Interlacing theorem,
$\lambda_{j}(\tilde{f}\left( \theta  \right))\leq\lambda_j(P_1)<\lambda_j(P_2)\leq \lambda_{j+1}(\tilde{f}\left( \theta  \right)), \,\,\textrm{ for all }\,\, \theta$.
Since the spectral gaps of $A$ are of size less than $\epsilon>0,$
the supremum of $\lambda_{j}(\tilde{f}\left( \theta  \right))+\frac{\epsilon}{2}$ is greater than or equal to the infimum of
$\lambda_{j+1}(\tilde{f}\left( \theta  \right))-\frac{\epsilon}{2}$. Therefore we have $\lambda_j(P_2)-\lambda_j(P_1)\leq \epsilon.$ 
Hence we get $\left|\lambda_i(P_2)-\lambda_i(P_1)\right|\leq \epsilon$ for every $i=1,2,\ldots p-1.$
\begin{equation}\nonumber
\hbox{Therefore,}\left| \hbox{trace}(P_1)-\hbox{trace}(P_2)\right|=\left|v_1-v_p\right|\leq (p-1)\epsilon. 
\end{equation}
Also since $v_1\leq v_2 \ldots \leq v_p,$ we get the desired conclusion with $c=v_j$ for some $j$.
\end{proof}
 \begin{remark}
The converse of the above result need not be true. That means if there exist $c\in \mathbb{R}$
such that $\displaystyle{\sup_{k\in \mathbb{Z}}|v_k-c|\leq \epsilon}$, then spectral gaps need not be small enough.
\end{remark}
\begin{remark}
Notice that we can not remove the assumption that $v_1\leq v_2\leq \ldots \leq v_p,$ since here we do not have the freedom to choose different symbols. 
\end{remark}
\section{Spectral Approximation Problems}\label{appl}
Consider a discrete  Schr$\ddot{\mbox{o}}$dinger operator $A$ with non-periodic potential.
Further assume that there exists a sequence of periodic discrete  Schr$\ddot{\mbox{o}}$dinger operators $A_n$ converges 
to $A$ in norm. Approximating the spectrum of $A$ using the spectral information of $A_n$ is an important problem in 
Mathematical Physics. In this section, we construct such approximations and sketch a different proof of the well-known
Ten Martini conjecture in view of the results discussed in the previous sections. We expect that a rigorous 
proof can be obtained by developing this sketch in the recent future. The numerical algorithms used in \cite{Embree}
to approximate the spectrum of quasiperiodic Jacobi operators give a positive sign in this direction.

The following two lemmas establish the connection between spectrum
of $A_n$ and $A$. These results are folklore and we present it here for the sake of completeness. $BL(\mathbb{H})$ denotes 
the algebra of all bounded linear operators on the Hilbert space $\mathbb{H}$. 
\begin{lemma}
Let $A_n\in BL(\mathbb{H})$ and $A_n\rightarrow A$ in $BL(\mathbb{H})$.
If for each $n\in \mathbb{N},$ $\lambda_n\in\sigma(A_n)$ and $\lambda_n\rightarrow \lambda$ as $n\rightarrow \infty,$
then $\lambda\in \sigma(A)$.
\end{lemma}
\begin{proof}
We have $\left\|A_n-\lambda_nI-(A-\lambda I)\right\|\leq \left\|A_n-A\right\|+ |\lambda_n - \lambda|\rightarrow 0$ as $n\rightarrow \infty.$ Therefore the sequence of singular operators $A_n-\lambda_nI$ converges to $A-\lambda I$. Since the singular elements in $BL(\mathbb{H})$ is a closed set we have $A-\lambda I$ is not invertible. Hence $\lambda\in \sigma(A)$.
\end{proof}
The following lemma shows that the converse is also true whenever we consider a sequence of normal operators on $BL(\mathbb{H})$.
\begin{lemma}
Let $A_n\in BL(\mathbb{H})$ and $A_n\rightarrow A$ in $BL(\mathbb{H})$. Further assume that each $A_n$ are normal.
Then the spectrum  of $A_n$ converges to the spectrum of $A$. That is the set
$\sigma(A) = \{\lambda;\lambda =\lim_{n\rightarrow \infty}\lambda_n;\; \lambda_n\in \sigma(A_n)\}$.
\end{lemma}
\begin{proof}
Since each $A_n$ is normal, $A$ is also normal. Therefore every spectral value is an approximate eigenalue.
That means $\lambda \in \sigma(A)$ if and only if there exists a sequence $x_n\in \mathbb{H}$ ( Weyl sequence) 
with $\rVert x_n\Arrowvert=1$ and
$\rVert Ax_n-\lambda x_n\Arrowvert \rightarrow 0$ as $n\rightarrow \infty.$ Now suppose that $\lambda \in \sigma(A)$ and $x_n$
be the corresponding Weyl sequence.
Then by Krylov-Weinstein inequality (see Proposition $1.41 (c) $ in \cite{Lim} for eg.), there exists a sequence
$\lambda_n \in \sigma(A_n)$ such that
\[
 \lvert \lambda_n - \lambda \rvert\leq \rVert A_nx_n-\lambda x_n\Arrowvert \,\,\,\,\textrm{for all} \,\,\,n.
\]
Therefore we have
\[
  \lvert \lambda_n - \lambda \rvert\leq \rVert A_nx_n-Ax_n+Ax_n -\lambda x_n\Arrowvert \leq  \rVert A_nx_n-Ax_n\Arrowvert+
  \rVert Ax_n-\lambda x_n\Arrowvert 
\]
Since $\rVert A_nx_n-Ax_n \Arrowvert \leq \rVert A_n-A \Arrowvert$ and 
$\rVert A_n-A\Arrowvert, \rVert Ax_n-\lambda x_n\Arrowvert$
tends to $0$ as $n \rightarrow \infty$, the right hand side of the above inequality tends to $0$ as $n \rightarrow \infty$.

Hence we get $\lambda_n \in \sigma(A_n)$ with $\lambda=\lim_{n\rightarrow \infty}\lambda_n.$ Hence the proof is completed.

\end{proof}
\begin{corollary}
Let $A_n\in BL(\mathbb{H})$ and $A_n\rightarrow A$ in $BL(\mathbb{H})$. Further assume that each $A_n$ are self-adjoint. Then the Hausdorff distance $d_H(\sigma(A), \sigma(A_{n}))$ between the spectrum of $A$ and the spectrum of $A_n$ converges to $0$ as $n$ tends to infinity.
\end{corollary}

\subsection{Ten Martini Problem}

 The well-known Ten Martini conjecture asserted by Barry Simon states the following; If we consider the almost periodic
 Mathieu potential say, $v_j=cos(2\pi j \alpha)$, where $\alpha$ is an irrational number, 
 then the spectrum of the associated discrete Schr$\ddot{\text{\textnormal{o}}}$dinger operator is a Cantor-like set.
 This conjecture was settled and many modified proofs are also available in the literature (see \cite{Avila} and references there
 in for more details). Most of the proofs involve deep number theoretic techniques as far as we know. 
 One possible operator theoretic approach to this problem was mentioned in \cite{kvbafa} using the approximation
 by sequence of periodic Schr$\ddot{\text{\textnormal{o}}}$dinger operators. We explain this below.

 Consider a sequence of rational numbers  $\alpha_n$ converges to $\alpha$. 
 The operators $A_n$ with potential $v_j^n=cos(2\pi j \alpha_n)$ will converge to the almost Mathieu operator $A$. 
 Even though $A$ is not periodic, each $A_n$ is periodic and will have spectral gaps for each $n$. 
 If the size of spectral gaps of $A_n$ decreases and number of gaps increases as $n$ increases, then we have large number
 of smaller spectral gaps for large $n$. The spectrum of $A_n$ will be obtained after removing these spectral gaps
 from the interval. As in the construction of Cantor set, after each stage we are removing more number of open intervals. 
 Proceeding in this way, we expect a much simpler proof for Ten Martini conjecture.

 The approximation by sequence of periodic Schr$\ddot{\text{\textnormal{o}}}$dinger operator is a crucial step in this approach. 
 How to choose an appropriate sequence is a challenging issue. We can choose the approximating
 sequence $A_n$ with periodic potentials $v_k^n$ in such a way that the sequence of periods $p_n$ 
 non decreasing. This can be proved without much difficulty. We end this section by proving this.
 
 \begin{theorem}
 Let $A$ be the discrete Schr$\ddot{o}$dinger operator with almost Mathieu potential $cos(2\pi j\alpha); \alpha$ irrational. Then there exists a sequence $A_n$ of discrete Schr$\ddot{o}$dinger operators with periodic potential $v_j^n$ such that $\left\|A_n-A\right\|\rightarrow 0$ as $n\rightarrow \infty$ and the sequence $p_n$ of periods of $v_j^n$ is non decreasing.
 \end{theorem}
 \begin{proof}
 Consider a sequence $\alpha_n$ of rationals converges to $\alpha$ and let $A_n$ be the discrete 
 Schr$\ddot{\text{\textnormal{o}}}$dinger operator with potential $v_j^n=cos(2\pi j\alpha_n).$ Then it is easy to verify that
 $\left\|A_n-A\right\|\rightarrow 0$ as $n\rightarrow \infty.$ We will show that the sequence  $\alpha_n$ can be chosen in such
 a way that sequence $p_n$ of periods of $v_j^n$ is non decreasing.
 
 \textbf{ \underline{Claim I:  }} If $\alpha_n=\frac{a_n}{b_n},$ then the minimal period of $v_j^n$ is $b_n+1.$\\
 It is easy to verify that $b_n+1$ is a period of $v_j^n.$ For we have,
 \[
 v_{b_n+1}^n=cos(2\pi(b_n+1)\frac{a_n}{b_n})=cos(2\pi a_n+2\pi \frac{a_n}{b_n})=cos(2\pi \frac{a_n}{b_n})=v_1^n.
 \]
 To show that $b_n+1$ is the minimal period, we list the sequence as $v_1^n, v_2^n, \ldots v_{b_n+1}^n=v_1^n.$ If there exists some $j$ with $2\leq j<b_{n}+1,$ and $v_1^n=v_j^n,$ then $cos(2\pi \frac{a_n}{b_n})=cos(2\pi j\frac{a_n}{b_n}).$ Therefore, $2\pi j\frac{a_n}{b_n}=2\pi \frac{a_n}{b_n}+2\pi k,$ for some natural number $k.$ Hence $j\frac{a_n}{b_n}=\frac{a_n}{b_n}+k$ and $(j-1)\frac{a_n}{b_n}=k, j=2,3,\ldots b_n.$ That means the set $\{\frac{a_n}{b_n}, 2\frac{a_n}{b_n},\ldots (b_n-1)\frac{a_n}{b_n}\}$ contains at least one natural number. This is not possible since $a_n$ and $b_n$ have no common factors.
  
\textbf{ \underline{Claim II:  }} $\alpha_n=\frac{a_n}{b_n},$ can be chosen in such a way that $b_n\leq b_{n+1}.$\\
Consider $\{\frac{1}{b_1}, \frac{1}{b_2},\ldots \}$ then it is a bounded sequence and therefore it has a convergent subsequence say $(\frac{1}{{b_n}^{'}}).$ Also, it has a monotonically decreasing subsequence which we denote by the same notation $(\frac{1}{{b_n}^{'}}).$ Now we consider $(\frac{{a_n}^{'}}{{b_n}^{'}})$ and obtain the desired conclusion. 
 \end{proof}
\begin{remark} 
We remark that using the above results, we are able to approximate almost Mathieu operator by a sequence of discrete
Schr$\ddot{\text{\textnormal{o}}}$dinger operators with non decreasing period. Once we are able to connect the number
of spectral gaps with the periodicity (that means if the number of spectral gaps increases as the periodicity increases),
then there is a possibility to obtain solution to Ten Martini problem in the method discussed above. 
\end{remark} 

\section{Concluding Remarks}
We conclude this article by pointing out some of the important aspects of the considered problem and listing out some
possible future problems. The results proved in section \ref{mainres} are the analogue of the results
in \cite{kiran} and \cite{KSN1}. In those articles only operators with connected spectrum were considered while here
 the operators under consideration are allowed to have spectral gaps.
 In section \ref{furthergen}, we could achieve much general and stronger results. 
 The major ingredients in the proofs are pure linear algebra and therefore the techniques are much simpler. 
 The connection between the main results and the Ten Martini problem discussed in section \ref{appl} pauses a major challenge;
 that is to obtain a much simpler proof of Ten Martini conjecture. Notice that this well-known operator theoretic question 
 has been addressed by several mathematicians and most of the proofs involve deep number theoretic techniques as far as we know. 
 The goal is to use the pure linear algebraic tools to prove this conjecture. 
 Here is the list of some open problems associated with this task.
\begin{itemize}
	\item Estimating the size of spectral gaps is a very important as well as challenging problem.
	\item If we can show that the number of spectral gaps increases as the periodicity increases,
	that will be a useful result. There is an affirmative example discussed in \cite{kiran} where the number of gaps is $p-1$
	($p$ is the period).
	\item The continuous case has to be addressed with the tools developed here. 
	Also, the pseudospectral version of the classical Borg's theorem is an interesting problem for future research.
\end{itemize}
 \subsection*{Acknowledgements}
The authors wish to thank Prof. M.N.N. Namboodiri and the anonymous referee for valuable suggestions.
V.B. Kiran Kumar is thankful to the University Grant Commission for the start-up research grant. 
 G. Krishna Kumar is thankful to the FLAIR programme of Higher Education Department, Govt. of Kerala and to Cochin University
 of Science And Technology (CUSAT) for supporting his visits to Department of Mathematics, CUSAT.
\bibliographystyle{amsplain}

\end{document}